\documentclass[reqno,twoside,a4paper,12pt]{amsart}

\usepackage{amsthm, amssymb, amsmath, amsfonts,eufrak}
\usepackage{url}
\usepackage[colorlinks=true,pdftex]{hyperref}
\hypersetup{linkcolor=red,citecolor=blue}
\usepackage{verbatim}
\usepackage{lineno}
\IfFileExists{epsf.def}{\input epsf.def}{\usepackage{epsf}}

\DeclareFontFamily{OT1}{eusb}{} \DeclareFontShape{OT1}{eusb}{m}{n} {<5> <6> <7> <8> <9> <10> <11> <12> <14.4> eusb10}{}
\DeclareMathAlphabet{\eusb}{OT1}{eusb}{m}{n}

\DeclareFontFamily{OT1}{eusm}{} \DeclareFontShape{OT1}{eusm}{m}{n} {<5> <6> <7> <8> <9> <10> <11> <12> <14.4> eusm10}{}
\DeclareMathAlphabet{\eusm}{OT1}{eusm}{m}{n}

\DeclareFontFamily{OT1}{eufm}{} \DeclareFontShape{OT1}{eufm}{m}{n} {<5> <6> <7> <8> <9> <10> <11> <12> <14.4> eufm10}{}
\DeclareMathAlphabet{\mathfrak}{OT1}{eufm}{m}{n}

\DeclareFontFamily{OT1}{fraktura}{}
\DeclareFontShape{OT1}{fraktura}{m}{n} {<5> <6> <7> <8> <9> <10> <11> <12> <13> <14.4> [1.1] eufm10}{}
\DeclareMathAlphabet{\fraktura}{OT1}{fraktura}{m}{n}

\DeclareFontFamily{OT1}{cmfi}{} \DeclareFontShape{OT1}{cmfi}{m}{n} {<5> <6> <7> <8> <9> <10> <11> <12> <13> <14.4> [0.9] cmfi10}{}
\DeclareMathAlphabet{\cmfi}{OT1}{cmfi}{b}{n}

\DeclareFontFamily{OT1}{cmss}{} \DeclareFontShape{OT1}{cmss}{m}{n} {<5> <6> <7> <8> <9> <10> <11> <12> <13> <14.4> cmss10}{}
\DeclareMathAlphabet{\cmss}{OT1}{cmss}{m}{n}

\newtheoremstyle{thm}{1.8ex}{1.8ex}{\itshape\rmfamily}{} {\bfseries\rmfamily}{}{2ex}{}

\newtheoremstyle{def}{1.8ex}{1.8ex}{\slshape\rmfamily}{} {\bfseries\rmfamily}{}{2ex}{}

\newtheoremstyle{rem}{1.8ex}{1.8ex}{\rmfamily}{} {\bfseries\rmfamily}{}{2ex}{}

\newenvironment{proofsect}[1] {\vspace{0.2cm}\noindent{\rmfamily\itshape#1.}}{\qed\vspace{0.15cm}}

\theoremstyle{thm}
\newtheorem{theorem}{Theorem}[section]
\newtheorem{lemma}[theorem]{Lemma}

\newtheorem*{Main Theorem}{Main Theorem.}

\newtheorem*{special theorem}{Lindeberg-Feller Theorem for Martingales}

\theoremstyle{def}

\newcommand\cour[1]{{\fontfamily{pcr}\selectfont #1}}

\theoremstyle{rem}
\newtheorem{remark}[theorem]{Remark}

\numberwithin{equation}{section}

\renewcommand{\section}{\secdef\sct\sect}
\newcommand{\sct}[2][default]{%
\refstepcounter{section}
\addcontentsline{toc}{section}{{\tocsection {}{\thesection}{\!\!\!\!#1\dotfill}}{}}
\vspace{0.7cm}
\centerline{\scshape\thesection.\ #1} \nopagebreak \vspace{0.2cm}}
\newcommand{\sect}[1]{%
\vspace{0.4cm} \centerline{\large\scshape\rmfamily #1}
\vspace{0.2cm}}

\renewcommand{\subsection}{\secdef\subsct\sbsect}
\newcommand{\subsct}[2][default]{\refstepcounter{subsection}
\addcontentsline{toc}{subsection}
{{\tocsection{\!\!}{\hspace{1.2em}\thesubsection}{\!\!\!\!#1\dotfill}}{}}
\nopagebreak\vspace{0.45\baselineskip} {\flushleft\bf
\thesubsection~\bf #1.~}
\\*[3mm]\noindent
\nopagebreak}
\newcommand{\sbsect}[1]{\vspace{0.1cm}\noindent
\textbf{#1.~}\vspace{0.1cm}}

\renewcommand{\subsubsection}{%
\secdef \subsubsect\sbsbsect}
\newcommand{\subsubsect}[2][default]{%
\refstepcounter{subsubsection} 
\addcontentsline{toc}{subsubsection}{{\tocsection{\!\!}
{\hspace{3.05em}\thesubsubsection}{\!\!\!\!#1\dotfill}}{}}
\nopagebreak
\vspace{0.15\baselineskip} \nopagebreak {\flushleft\rmfamily
\itshape\thesubsubsection
\ \rmfamily #1\/.}\ }
\newcommand{\sbsbsect}[1]{\vspace{0.1cm}\noindent
\rmfamily \itshape
\arabic{section}.\arabic{subsection}.\arabic{subsubsection} \
\sffamily #1\/.\ }

\renewcommand{\caption}[1]{%
\vglue0.5cm
\refstepcounter{figure}
\begin{minipage}{0.9\textwidth}\small {\sc Figure~\thefigure. }#1\end{minipage}}

\newcommand{\twocite}[2]{\cite{#1}--\cite{#2}}
\newcommand{\N}        {\mathbb N}
\newcommand{\R}        {\mathbb R}
\newcommand{\B}        {\mathbb B}

\newcommand{\Z}        {\mathbb Z}
\newcommand{\Q}        {\mathbb Q}

\newcommand{\AAA}         {\mathcal{A}} 
 
\newcommand{\CC}          {\mathcal{C}} 

\newcommand{\BB}          {\mathcal{B}}

\newcommand{\LL}         {\mathcal{L}}

\newcommand{\Prob}        {\mathbb P}

\newcommand{\twoeqref}[2]{(\ref{#1}--\ref{#2})}

\def\myffrac#1#2 in #3{\raise 2.6pt\hbox{$#3 #1$}\mkern-1.5mu\raise 0.8pt\hbox{$#3/$}\mkern-1.1mu\lower 1.5pt\hbox{$#3 #2$}}
\newcommand{\ffrac}[2]{\mathchoice%
    {\myffrac{#1}{#2} in \scriptstyle}
    {\myffrac{#1}{#2} in \scriptstyle}
    {\myffrac{#1}{#2} in \scriptscriptstyle}
    {\myffrac{#1}{#2} in \scriptscriptstyle}
}

\title[Random Walk among random conductances]{Heat-kernel estimates for random walk among random conductances with heavy tail}
\author[Omar Boukhadra
]{Omar BOUKHADRA$^*$}
\address{ CMI, 39 rue F. Joliot-Curie     13453 Marseille cedex 13, France.\newline
\indent
D\'epartement de math\'ematiques, Universit\'e de Constantine, BP 325, route Ain El Bey, 25017, Constantine, Alg\'erie.} 
\begin{document} 
\thanks{\hglue-4.5mm\fontsize{9.6}{9.6}\selectfont\copyright\,2009 by Omar Boukhadra. Provided for non-commercial research and education use. Not for reproduction, distribution or commercial use.\vspace{2mm} \\\cour{$^*$E-mail address~: omar.boukhadra@cmi.univ-mrs.fr}}
\maketitle

\vspace{-4mm}
\centerline{\textit{Centre de Math\'ematiques et Informatique (CMI),}}
\centerline{\textit{Universit\'e de Provence};}
\centerline{\textit{D\'epartement de Math\'ematiques, Universit\'e de Constantine}}

\vspace{-2mm}
\begin{abstract}
We study models of discrete-time, symmetric, $\Z^{d}$-valued random walks in
random environments, driven by a field of i.i.d. random nearest-neighbor conductances $\omega_{xy}\in[0,1]$, with polynomial tail near 0 with exponent $\gamma>0$. We first prove for all $d\geq5$ that the return probability shows an anomalous decay (non-Gaussian) that approches (up to sub-polynomial terms) a random constant times $n^{-2}$ when we push the power $\gamma$ to zero. In contrast, we prove that the heat-kernel decay is as close as we want, in a logarithmic sense, to the standard decay $n^{-d/2}$ for large values of the parameter $\gamma$. 
\newline\\
\noindent
\textit{\textbf{keywords}}~:
Random walk,  Random environments,  Markov chains,
Random conductances,  Percolation.
\end{abstract}

\section{\textbf{Introduction and results}}
\label{}
The main purpose of this work is the derivation
of heat-kernel bounds for random walks $(X_n)_{n\in\N}$ among polynomial lower tail random conductances with exponent $\gamma>0$, on $\Z^d, d>4$. We show that the heat-kernel exhibits opposite behaviors, anomalous and standard, for small and large values of $\gamma$.

Random walks in reversible random environments are driven by the transition matrix
\begin{equation}
\label{protra}
P_{\omega}(x,y)=\frac{\omega_{xy}}{\pi_{\omega}(x)}.
\end{equation}
where $(\omega_{xy})$ is a family of random (non-negative) conductances subject to the symmetry condition~$\omega_{xy}=\omega_{yx}$. The sum $\pi_\omega(x)=\sum_y\omega_{xy}$ defines an invariant, reversible measure for the corresponding discrete-time Markov chain. 
In most situations~$\omega_{xy}$ are non-zero only for nearest neighbors on~$\Z^d$ and are sampled from a shift-invariant, ergodic or even i.i.d.\ measure~$\Q$.

One general class of results is available for such random walks under the additional assumptions of uniform ellipticity,
\begin{displaymath}
\exists\alpha>0:\quad \Q(\alpha<\omega_{b}<1/\alpha)=1
\end{displaymath}   
and the boundedness of the jump distribution,
\begin{displaymath}
\exists R<\infty:\, \vert x\vert\geq R\, \Rightarrow \, P_{\omega}(0,x)=0,\quad \Q-a.s.
\end{displaymath}
One has then the standard local-CLT like decay of the heat-kernel ($c_1,c_2$ are absolute constants), as proved by Delmotte~\cite{del}:
\begin{equation}
\label{heat-kernel}
 P^{n}_{\omega}(x,y)\leq\frac{c_{1}}{n^{d/2}}\exp\left \{-c_{2}\frac{\vert x-y\vert^{2}}{n}\right\}.
\end{equation}

Once the assumption of uniform ellipticity is relaxed, matters get more complicated. The most-intensely studied example is the simple random walk on the infinite cluster of supercritical bond percolation on~$\Z^d$, $d\ge2$. This corresponds to~$\omega_{xy}\in\{0,1\}$ i.i.d. with~$\Q(\omega_b=1)>p_c(d)$ where~$p_c(d)$ is the percolation threshold (cf. \cite{G}). Here an annealed invariance principle has been obtained by De Masi, Ferrari, Goldstein and Wick~\twocite{demas1}{demas2} in the late 1980s. More recently, Mathieu and R\'emy~\cite{Mathieu-Remy} proved the on-diagonal (i.e., $x=y$) version of the heat-kernel upper bound \eqref{heat-kernel}---a slightly weaker version of which was also obtained by Heicklen and Hoffman~\cite{Heicklen-Hoffman}---and, soon afterwards, Barlow~\cite{Barlow} proved the full upper and lower bounds on 
$P_\omega^n(x,y)$ of the form \eqref{heat-kernel}. (Both these results hold for $n$ exceeding some random time defined relative to the environment in the vicinity of~$x$ and~$y$.) Heat-kernel upper bounds were then used in the proofs of quenched invariance principles by Sidoravicius and Sznitman~\cite{Sidoravicius-Sznitman} for $d\ge4$, and for all $d\ge2$ by Berger and Biskup~\cite{BB} and Mathieu and Piatnitski~\cite{Mathieu-Piatnitski}.

We consider in our case a family of symmetric, irreducible, nearest-neighbor Markov chains on~$\Z^d$, $d\ge5$, driven by a field of i.i.d. bounded random conductances $\omega_{xy}\in[0,1]$ and subject to the symmetry condition~$\omega_{xy}=\omega_{yx}$. These are constructed as follows. Let $\Omega$ be the set of functions $\omega:\Z^d\times\Z^d \rightarrow \R_{+}$ such that $\omega_{xy}>0$ iff $x \sim y$, and  $\omega_{xy}=\omega_{yx}$ ( $x \sim y$ means that $x$ and $y$  are nearest neighbors). We call elements of $\Omega$ environments.

We choose the family $\{ \omega_{b}, b=(x,y),x \sim y, b\in\Z^d\times\Z^d\}$ i.i.d according to a law $\Q$ on $(R^{\ast}_{+})^{\Z^d}$ such that 
\begin{equation}
\label{1}
\begin{array}{ll}
\omega_{b}\leq 1 & \text{for all } b;\\
\Q(\omega_{b} \leq a)\sim a^{\gamma} & \text{when } a\downarrow 0,
\end{array}
\end{equation}
where $\gamma>0$ is a parameter. Therefore, the conductances are $\Q$-a.s. positive.

In a recent paper, Fontes and Mathieu~\cite{Fontes-Mathieu} studied continuous-time random walks on~$\Z^d$ which are defined by  generators~$\LL_\omega$ of the form
\begin{displaymath}
(\mathcal{L}_{\omega}f)(x)=\sum_{y\sim x}\omega_{xy}[f(y)-f(x)],
\end{displaymath}  
with conductances given by
\begin{displaymath}
\omega_{xy}=\omega(x)\wedge \omega(y)
\end{displaymath}
for i.i.d.\ random variables~$\omega(x)>0$ satisfying \eqref{1}. For these cases, it was found that the annealed heat-kernel, $\int \text{d}\Q(\omega) P^{\omega}_{0}(X_{t}=0)$, exhibits an \emph{anomalous decay}, for $\gamma< d/2$. Explicitly, from \cite{Fontes-Mathieu}, Theorem 4.3, we have
\begin{equation}
\label{fms}
\int \text{d}\Q(\omega) P^{\omega}_{0}(X_{t}=0)=t^{-(\gamma\wedge\frac{d}{2})+o(1)}, \quad t\rightarrow\infty.
\end{equation}

In addition, in a more recent paper, Berger, Biskup, Hoffman and Kozma \cite{berger}, provided universal upper bounds on the quenched heat-kernel by considering the nearest-neighbor simple random walk on~$\Z^d$, $d\ge2$, driven by a field of i.i.d. bounded random conductances $\omega_{xy}\in[0,1]$. The conductance law is i.i.d.\ subject to the condition that the probability of $\omega_{xy}>0$ exceeds the threshold $p_c(d)$  for bond percolation on~$\Z^d$. For environments in which the origin is connected to infinity by bonds with positive conductances, they studied the decay of the $2n$-step return probability $P_\omega^{2n}(0,0)$. They have proved that $P_\omega^{2n}(0,0)$ is bounded by a random constant times $n^{-d/2}$ in $d=2,3$, while it is $o(n^{-2})$ in~$d\ge5$ and $O(n^{-2}\log n)$ in $d=4$. More precisely, from \cite{berger}, Theorem 2.1, we have for almost every $\omega\in\{0\in \mathcal{C}_{\infty}\}$ ($\mathcal{C}_{\infty}$ represents the set of sites that have a path to infinity along bonds with positive conductances), and for all $n\geq1$. 
\begin{equation}
\label{trans}
P_\omega^n(0,0)\le C(\omega)\,
\begin{cases}
n^{-d/2},\qquad&d=2,3,
\\
n^{-2}\log n,\qquad&d=4,
\\
n^{-2},\qquad&d\ge5,\end{cases}
\end{equation}
where $C(\omega)$ is a random positive variable.\\
On the other hand, to show that those general upper bounds (cf. \eqref{trans}) in $d\geq5$ represent a real phenomenon, they  produced examples with anomalous heat-kernel decay approaching $1/n^2$, for i.i.d. laws $\Q$ on bounded nearest-neighbor conductances with \textit{lower tail much heavier than polynomial} and with~$\Q(\omega_b>0)>p_c(d)$. We quote Theorem 2.2 from \cite{berger}~:
\begin{theorem}
\label{thm2}
(1)
Let~$d\ge5$ and $\kappa>1/d$. There exists an i.i.d.\
law~$\Q$ on bounded, nearest-neighbor conductances
with~$\Q(\omega_b>0)>p_c(d)$ and a random
variable~$C=C(\omega)$ such that for almost
every~$\omega\in\{0\in\mathcal{C}_\infty\}$,
\begin{equation}
\label{lower-bd}
P_\omega^{2n}(0,0)\ge\ C(\omega)\frac{\text e^{-(\log n)^\kappa}}{n^2},
\qquad n\ge1.
\end{equation}

\noindent
(2) Let $d\ge5$. For every increasing sequence
$\{\lambda_n\}_{n=1}^\infty$, $\lambda_n\to\infty$, there exists an i.i.d.\ law $\Q$ on
bounded, nearest-neighbor conductances
with~$\Q(\omega_b>0)>p_c(d)$ and an a.s.\ positive random variable~$C=C(\omega)$ such that for almost
every~$\omega\in\{0\in\mathcal{C}_\infty\}$,
\begin{equation}
\label{2.4}
P_\omega^n(0,0)\ge \frac{C(\omega)}{\lambda_nn^2}
\end{equation}
along a subsequence that does not depend on~$\omega$.
\end{theorem}
The distributions that they use in part~(1) of Theorem~\ref{thm2} have a tail near zero of the general form
\begin{equation}
\Q(\omega_{xy}<s) \approx |\log(s)|^{-\theta}
\end{equation}
with~$\theta>0$.

Berger, Biskup , Hoffman and Kozma \cite{berger} called attention  to the fact that the construction of an estimate of the anomalous heat-kernel decay for random walk among polynomial lower tail random conductances on $\Z^d$, seems to require subtle control of heat-kernel \emph{lower} bounds which go beyond the estimates that can be easily pulled out from the literature.
In the present paper, we give a response to this question and show that every distribution with an appropriate power-law decay near zero, can serve as such example, and that when we push
the power to zero. The lower bound obtained for the return probability approaches
(up to sub-polynomial terms) the upper bound supplied by \cite{berger} and that  for all $d\geq5$.

Here is our first main result whose proof is given in section~\ref{ahkd}~:
\begin{theorem}
\label{th}
Let $d\geq5$. There exists a positive constant $\delta(\gamma)$ depending only on $d$ and $\gamma$ such that $\Q$-a.s., there exists~$C=C(\omega)<\infty$ and  for all $n\geq1$  
\begin{equation} 
\label{min}
P^{2n}_{\omega}(0,0)\geq \frac{C}{n^{2+\delta(\gamma)}}\quad \text{and}\quad \delta(\gamma)\xrightarrow[\gamma \to 0]{}0.
\end{equation}
\end{theorem}

\begin{remark}
\label{rem}
\begin{enumerate}
  \item 
	The proof tells us in fact, with \eqref{trans}, that for $d\geq5$ we have almost surely 
  \begin{equation}
  \label{esup}
  \begin{split}
  & -2[1+d(2d-1)\gamma]\leq \liminf_{n} \frac{\log P^{2n}_{\omega}(0,0)}{\log n}\\
  &\qquad\qquad\qquad\qquad\qquad\qquad\leq \limsup_{n} \frac{\log P^{2n}_{\omega}(0,0)}{\log n}\leq -2.
  \end{split}
  \end{equation}
	\item As we were reminded by M. Biskup and T.M. Prescott, the invariance principle (CLT) (cf Theorem 2.1. in \cite{BP} and Theorem 1.3 in \cite{QIP}) automatically implies the ``usual'' lower bound on the heat-kernel under weaker conditions on the conductances. Indeed, the Markov property and reversibility of~$X$ yield 
	$$
	P^{\omega}_{0}(X_{2n}=0)\geq \frac{\pi_\omega(0)}{2d}\sum_{x\in \mathcal{C}_{\infty}\atop \vert x\vert\leq \sqrt{n}}P^{\omega}_{0}(X_{n}=x)^{2}.
	$$
	Cauchy-Schwarz then gives
	$$
	P^{\omega}_{0}(X_{2n}=0)\geq 	P^{\omega}_{0}(\vert X_{n}\vert \leq \sqrt{n})^{2}\frac{\pi_\omega(0)/2d}{\vert \mathcal{C}_{\infty}\cap [-\sqrt{n},+\sqrt{n}]^{d}\vert}.
	$$
	Now the invariance principle implies that $P^{\omega}_{0}(\vert X_{n}\vert \leq \sqrt{n})^{2}$ has a positive limit as~$n\to\infty$ and the Spatial Ergodic Theorem shows that $\vert \mathcal{C}_{\infty}\cap [-\sqrt{n},+\sqrt{n}]^{d}\vert$ grows proportionally to~$n^{d/2}$. Hence we get 
	$$
	P ^{\omega}_{0}(X_{2n}=0)\geq \frac{C(\omega)}{n^{d/2}}, \quad n\geq 1,
	$$
	with~$C(\omega)>0$ a.s. on the set~$\{0\in \mathcal{C}_{\infty}\}$. Note that, in $d=2,3$, this complements nicely the ``universal'' upper bounds derived in~\cite {berger}. In $d=4$, the decay is
at most $n^{-2}\log n$ and at least $n^{-2}$.
	
\end{enumerate}
\end{remark}

The result of Fontes and Mathieu \eqref{fms} (cf. \cite{Fontes-Mathieu}, Theorem 4.3) encourages us to believe that the quenched heat-kernel has a standard decay when $\gamma\geq d/2$, but the construction  seems to require subtle control of heat-kernel upper bounds. In the second result of this paper whose proof is given in section~\ref{shd}, we prove, for all $d\geq5$,  that the heat-kernel decay is as close as we want, in a logarithmic sense, to the standard decay $n^{-d/2}$ for large values of the parameter $\gamma$. For the cases where $d=2,3$, we have a standard decay of the quenched return probability under weaker conditions on the conductances (see Remark \ref{rem}).

\begin{theorem}
\label{thm}
Let $d\geq5$. There exists a positive constant $\delta(\gamma)$ depending only on $d$ and $\gamma$ such that $\Q$-a.s.,  
\begin{equation} 
\label{min}
\limsup_{n\rightarrow+\infty}\sup_{x\in\Z^d}\frac{\log P^{n}_{\omega}(0,x)}{\log n}\leq -\frac{d}{2}+\delta(\gamma)\quad \text{and}\quad \delta(\gamma)\xrightarrow[\gamma \to +\infty]{}0 .
\end{equation}
\end{theorem}

In what follows,, we refer to $P^{\omega}_{x}(\cdot)$ as the \textit{quenched} law of the random walk $X=(X_{n})_{n\geq 0}$ on $((\Z^d)^{\N}, \mathcal{G})$ with transitions given in \eqref{protra} in the environment~$\omega$, where $\mathcal{G}$ is the $\sigma-$algebra generated by cylinder functions, and let $\mathbb{P}:=\Q\otimes P^\omega_0$ be the
so-called \textit{annealed} semi-direct product measure law defined by
$$
\Prob(F\times G)=\int_F \Q(\text{d}\omega)P^\omega_0(G), \quad F\in \mathcal{F}, G\in \mathcal{G}.
$$
where $\mathcal{F}$ denote the Borel $\sigma-$algebra on $\Omega$ (which is the same as the $\sigma-$algebra generated by cylinder functions).

\section{\textbf{Anomalous heat-kernel decay}}
\label{ahkd}
In this section we provide the proof of Theorem \ref{th}.

We consider a family of bounded nearest-neighbor conductances~$(\omega_b)\in\Omega=[0,1]^{\B^d}$ where~$b$ ranges over the set~$\B^d$ of unordered pairs of nearest neighbors in~$\Z^d$. The law $\Q$ of the~$\omega$'s will be i.i.d.\ subject to the conditions given in \eqref{1}.\\

We prove this lower bound by following a different approach of the one ado\-pted by Berger, Biskup , Hoffman and Kozma \cite{berger} to prove \twoeqref{lower-bd}{2.4}. In fact, they prove that in a box of side length~$\ell_n$ there
exists a configuration where a strong bond with conductance of order 1, is separated from other sites by bonds of strength~$1/n$, and (at least) one of these ``weak'' bonds is connected to the origin by a ``strong'' path not leaving the box. Then the probability that the walk is back to the origin at time~$n$ is bounded below by the probability that the walk goes directly
towards the above pattern (this costs $ e^{O(\ell_n)}$ of
probability) then crosses the weak bond (which costs~$1/n$),
spends time $n-2\ell_n$ on the strong bond (which costs only $O(1)$
of probability), then crosses a weak bond again (another factor
of~$1/n$) and then heads towards the origin to get there on
time (another $ e^{O(\ell_n)}$ term). The cost of this
strategy is $O(1) e^{O(\ell_n)}n^{-2}$ so if $\ell_n=o(\log
n)$ then we get leading order~$n^{-2}$.\\

Our method for proving Theorem \ref{th} is, in fact, simple - we note that due to the reversibility of
the walk and with a good use of Cauchy-Schwartz, one does not need to condition
on the exact path of the walk, but rather show that the walker has a relatively large
probability of staying within a small box around the origin. Our  objective will consist in showing that for almost every~$\omega$, the probability that the random walk when started at the origin is at time~$n$ inside the box~$B_{n^{\delta}}=[-3n^{\delta},3n^{\delta}]^{d}$, is greater than~$c/n$ (where $c$ is a constant and  $\delta=\delta(\gamma)\downarrow 0$). Hence we will get $P^{2n}_{\omega}(0,0)/\pi(0) \geq c/n^{2+\delta d}$ by virtue of  the following inequality which, for almost every environment~$\omega$, derives from the reversibility of $X$, Cauchy-Schwarz inequality and  \eqref{1}~:

\begin{eqnarray}
\label{minun}
\frac{P^{2n}_{\omega}(0,0) }{\pi_\omega(0)}
&\geq& 
\sum_{y\in B_{n^{\delta}}}\frac{P^{n}_{\omega}(0,y)^{2}}{\pi_\omega(y)} 
\nonumber\\ 
&\geq&
\left(\sum_{y\in B_{n^{\delta}}}P^{n}_{\omega}(0,y)\right)^{2} \frac{1}{\pi_\omega(B_{n^{\delta}})}
\nonumber\\
&\geq&
\frac{P^{\omega}_{0}(X_{n}\in B_{n^{\delta}})^{2}}{\# B_{n^{\delta}}}.
\end{eqnarray}
In order to do this, our strategy  is to show that the random walk meets a \textit{trap}, with positive probability, before getting out from $[-3n^{\delta},3n^{\delta}]^{d}$, where, by
definition, a trap is an edge of conductance of order $1$ that can be
reached only by crossing an edge of order $1/n$. The random walk, being imprisoned in the trap inside the box $[-3n^{\delta},3n^{\delta}]^{d}$,  will not get out from this box before time~$n$ with positive probability. Then the Markov property yields $P^{\omega}_{0}(X_n\in [-3n^{\delta},3n^{\delta}]^{d})\geq c/n$. Thus, we will be brought to  follow the walk until it finds a specific configuration in the environment.

First, we will need to prove one lemma. Let $B_{N}=[-3N,3N]^{d}$ be the box centered at the origin and of radius $3N$ and define $\partial B_{N}$ to be its inner boundary, that is, the set of vertices in $B_N$ which are adjacent to some vertex not in $B_N$. We have $\#B_N\leq (7N)^{d}$. Let~$H_{0}=0$ and  define $H_{N}$, $N\geq1$, to be the hitting time of $\partial B_{N}$, i.e.
\begin{displaymath}
H_{N}=\inf \{n\geq0:X_{n}\in \partial B_{N}\}.
\end{displaymath} 
The box~$B_{N}$ being finite for $N$ fixed, we have  then $H_{N}<\infty$ a.s., \mbox{$\forall N\geq 1$.} 

Let $\hat{e}_{i}, \, i=1,\ldots, d$, denote the canonical unit vectors in $\Z^{d}$, and let $x\in \Z^{d}$, with $x:=(x_{1},\ldots,x_{d})$. Define $i_{0}:=\max\{i:\vert x_{i}\vert\geq\vert x_{j}\vert, \forall j\neq i\}$ and let $\epsilon (x): \Z^{d}\rightarrow \{-1,1\}$ be the function such that
\begin{displaymath}
\epsilon (x)= \begin{cases}
+1 & \text{if } x_{i_{0}}\geq 0 \\
-1 & \text{if } x_{i_{0}}<0 \end{cases} 
\end{displaymath}
Now, let $\alpha, \xi$ be positive constants such that $\Q(\omega_{b}\geq\xi)>0$. Define $\AAA_{N}(x)$ to be the event that the configuration near $x, y=x+\epsilon(x)\hat{e}_{i_{0}}$ and  $z=x+2\epsilon(x)\hat{e}_{i_{0}}$ is as follows:
\begin{enumerate}
	\item $\frac{1}{2} N^{-\alpha}<	\omega_{xy}\leq  N^{-\alpha}$.
	\item $\omega_{yz}\geq\xi$.
	\item every other bond emanating out of $y$ or $z$ has $\omega_{b}\leq N^{-\alpha}$. 
\end{enumerate}
The event~$\AAA_{N}(x)$ so constructed involves a collection of $4d-1$ bonds that will be denoted by $\CC(x)$, i.e.
\begin{eqnarray*}
\begin{split}
& \CC(x):=\{[x,y],[y,z],[y,y^i],[z,z^i],[z,z^i_0]; y=x+\epsilon(x)\hat{e}_{i_{0}},z=x+2\epsilon(x)\hat{e}_{i_{0}},\\
& \qquad \qquad \qquad \qquad\qquad \qquad y^i=y\pm\hat{e}_{i}, z^i=z\pm\hat{e}_{i},\forall i\neq i_0,z^i_{0}=z+\epsilon(x)\hat{e}_{i_0}  \}
\end{split}
\end{eqnarray*}

Let us note that if $x\in \partial B_N$, for some $N\geq 1$, the collection $\CC(x)$ is outside the box~$B_N$ and if $y\in \partial B_K$, for $K\neq N$, we have $\CC(x)\cap \CC(y)=\emptyset$. \\
If the bonds of the collection $\CC(x)$ satisfy the conditions of the event $\AAA_{N}(x)$, we agree to call it a \textit{trap} that we will denote by $\mathfrak{P}_{N}$.

The lemma says then that~:
\begin{lemma}
\label{I}
The family $\{\AAA^{k}_{N}=\AAA_{N}(X_{H_{k}})\}^{N-1}_{k=0}$ is $\mathbb{P}$-independent for each $N$.
\end{lemma}
\begin{proof}
The occurrence of the event $\AAA_{N}(X_{H_{k}})$ means that the random walk $X$ has met a trap~$\mathfrak{P}_{N}$ situated outside of the box~$B_{k}$ when it has hit for the first time the boundary of the box~$B_{k}$.

Let $q_{N}$ be the $\Q$-probability of having the configuration of the trap $\mathfrak{P}_{N}$. We have $q_{N}=\Q(\AAA_{N}(x))=\mathbb{P}[\AAA_{N}(X_{H_{k}})],\, \forall x\in \partial B_{k}$ and $\forall k\leq N-1$. Indeed, by virtue of the i.i.d. character of the conductances and the Markov property, when the random walk hits the boundary of $B_{k}$ for the first time at some element~$x$, the probability that the collection $\CC(x)$ constitutes a trap, i.e., satisfies the conditions of the event $\AAA_N(x)$, depends only on the edges of the collection $\CC(x)$, which have not been visited before. \\ 
Let $k_{1}< k_{2}\leq N-1$ and $x\in \partial B_{k_{2}}$, we have then
\begin{eqnarray*}
\Prob\left[\AAA^{k_{1}}_{N}, X_{H_{k_{2}}}=x,\AAA^{k_{2}}_{N}\right]
&=&
\Prob\left[\left\{\AAA^{k_{1}}_{N}, X_{H_{k_{2}}}=x\right\}\cap\AAA_N(x)\right]\\
&=&
\Prob\left[\AAA^{k_{1}}_{N}, X_{H_{k_{2}}}=x\right]\Prob\left[\AAA_N(x)\right]\\
&=&
q_N\Prob\left[\AAA^{k_{1}}_{N}, X_{H_{k_{2}}}=x\right],
\end{eqnarray*}
since the events $\{\AAA^{k_{1}}_{N}, X_{H_{k_{2}}}=x\}$ and $\AAA_{N}(x)$ depend respectively on the conductances of the bonds of $B_{k_{2}}$ and the conductances of the bonds of the collection $\CC(x)$ which is situated outside the box $B_{k_{2}}$ when $x\in \partial B_{k_{2}}$.

Thus
\begin{eqnarray*}
\mathbb{P}\left[\AAA^{k_{1}}_{N}\AAA^{k_{2}}_{N}\right]
&=&
\sum_{x\in \partial B_{k_{2}}}\Prob\left[\AAA^{k_{1}}_{N}, X_{H_{k_{2}}}=x,\AAA^{k_{2}}_{N}\right]
\\
&=&
q_{N}\sum_{x\in \partial B_{k_{2}}}\Prob\left[\AAA^{k_{1}}_{N}, X_{H_{k_{2}}}=x\right]
\\
&=&
q_{N}\mathbb{P}\left[\AAA^{k_{1}}_{N}\right]=q^{2}_{N}.
\end{eqnarray*}

With some adaptations, this reasoning remains true in the case of more than two events $\AAA^{k}_{N}$.
\end{proof}

We come now to the proof of Theorem~\ref{th}.\\

\begin{proofsect}{Proof of Theorem~\ref{th}}
Let $d\geq5$ and $\gamma>0$. Set $\alpha=\frac{1-\epsilon}{(4d-2)\gamma}$ for arbitrary positive constant $\epsilon<1$ (the constant $\alpha$ is the same used in the definition of the event $\AAA_ N(x)$). As seen before (cf. \eqref{minun}), for almost every environment~$\omega$,  the reversibility of $X$, Cauchy-Schwarz inequality and  \eqref{1} give

\begin{equation}
\label{minun2}
\frac{P^{2n}_{\omega}(0,0) }{\pi_\omega(0)}
\geq 
\frac{P^{\omega}_{0}(X_{n}\in B_{n^{1/\alpha}})^{2}}{\# B_{n^{1/\alpha}}},
\end{equation}

By the assumption \eqref{1} on the conductances and the definition of the event $\AAA_N	(x)$, the probability of having the configuration of the trap $\mathfrak{P}_{N}$ is greater than $cN^{-(1-\epsilon)}$ (where $c$ is a constant that we use henceforth as a generic constant). Indeed, when~$N$ is large enough, we have
\begin{eqnarray*}
q_{N}
&=&
 \Q\left(\frac{1}{2} N^{-\alpha}<	\omega_{xy}\leq  N^{-\alpha}\right)
 \Q(\omega_{yz}\geq\xi)
 \left[\Q(\omega_{b}\leq N^{-\alpha})\right]^{4d-3} 
\geq
\frac{c}{N^{1-\epsilon}}.
\end{eqnarray*}

Consider now the following event
$$
\Lambda_{N}:=\bigcup^{N-1}_{k=0}\AAA^{k}_{N}.
$$
The event~$\Lambda_{N}$ so defined may be interpreted as follows~: \textit{at least, one among the $N$ disjoint collections $\CC(X_{H_{k}}),\, k\leq N-1$, constitutes a trap $\mathfrak{P}_{N}$}. 
The events $\AAA^{k}_{N}$ being independent by lemma \ref{I}, we have
\begin{eqnarray}
\label{7} 
\mathbb{P}[\Lambda^{c}_{N}]
&\leq&
\left(1-cN^{\epsilon-1}\right)^{N} \nonumber
\\
&\leq&
\exp\left\{N\log\left(1-cN^{\epsilon-1}\right)\right\}\nonumber
\\
&\leq&
\exp\left\{-cN^{\epsilon}\right\}.
\end{eqnarray}
 Chebychev inequality and \eqref{7} then give
\begin{equation}
\label{cantelli}
\sum^{\infty}_{N=1}\Q\left\{\omega: P^{\omega}_{0}(\Lambda^{c}_{N})\geq 1/2\right\}
\leq 
2\sum^{\infty}_{N=1}\mathbb{P}[\Lambda^{c}_{N}]<+\infty.
\end{equation}
It results by Borel-Cantelli lemma that for almost every $\omega$, there exists $N_{0}\geq1$ such that for each $N\geq N_{0}$, the event $\AAA_{N}(x)$ occurs inside the box $B_{N}$ with positive probability (greater than~$1/2$) on the path of $X$, for some $x\in B_{N-1}$. For almost every~$\omega$, one may say that $X$ meets with positive probability a trap $\mathfrak{P}_{N}$ at some site $x\in B_{N-1}$ before getting outside of $B_{N}$.

Suppose that~$N\ge N_0$ and let~$n$ be such that~$N^{\alpha}\leq n<(N+1)^{\alpha}$. Define
$$
D_{N}:= \left\{
\begin{array}{ll}
\inf\{k\leq N-1: \AAA^{k}_{N}\,\text{occurs}\} & \text{if} \quad \Lambda_{N}\,\text{occurs}\\
+\infty & \text{otherwise},
\end{array}
\right.
$$ 
to be the rank of the first among the~$N$ collections $\CC(X_{H_{k}}),\, k\leq N-1$, that constitutes a trap $\mathfrak{P}_{N}$. If $D_{N}=k$, the random variable~$D_{N}$ so defined depends only on the steps of $X$ up to time~$H_{k}$. Thus, if $D_{N}=k$, we have  $X_{H_{k}}\in B_{N-1}$ and $\CC(X_{H_{k}})$ constitutes a trap $\mathfrak{P}_{N}$.
So, if we set $X_{H_{k}}=x$, the bond~$[x,y]$ (of the trap $\mathfrak{P}_{N}$) will have then a conductance of order $N^{-\alpha}$. In this case, the probability for the random walk, when started at~$X_{H_{k}}=x$, to cross the bond $[x,y]$ is by the property (1) of the definition of the event~$\AAA_N(x)$ above greater than
\begin{equation}
\label{b1}
\frac{(1/2)N^{-\alpha}}{\pi_{\omega}(x)}\geq \frac{1/2}{2dN^{\alpha}}= \frac{1}{4dN^{\alpha}}.
\end{equation} 
Here we use the fact that $\pi_{\omega}(x)\leq 2d$ by virtue of \eqref{1}.
This implies by the Markov property and by \eqref{b1} that
\begin{equation}
\label{if}
\begin{split}
& P^{\omega}_{0}(X_{n}\in B_{N}|D_{N}\leq N-1)\\ 
& \qquad =\sum^{N-1}_{k=0}\sum_{x\in B_{k}}\frac{P^{\omega}_{0}(X_{n}\in B_{N},D_{N}=k, X_{H_{k}}=x)}{P^{\omega}_{0}(D_{N}\leq N-1)}
\\
&\qquad\geq  
\sum^{N-1}_{k=0}\sum_{x\in B_{k}}\frac{P^{\omega}_{0}(H_{N}\geq n, D_{N}=k, X_{H_{k}}=x)}{P^{\omega}_{0}(D_{N}\leq N-1)}
\\
& \qquad \geq 
\sum^{N-1}_{k=0}\sum_{x\in B_{k}}\frac{P^{\omega}_{0}(D_{N}=k, X_{H_{k}}=x)}{P^{\omega}_{0}(D_{N}\leq N-1)}
P^{\omega}_{x}(H_{N}\geq n)\\
& \qquad \geq 
\sum^{N-1}_{k=0}\sum_{x\in B_{k}}\frac{P^{\omega}_{0}(D_{N}=k, X_{H_{k}}=x)}{P^{\omega}_{0}(D_{N}\leq N-1)} P^{\omega}_{y}(H_{N}\geq n)P^{\omega}_{x}(X_{1}=y)\\
&\qquad \geq 
\frac{1}{4dN^{a}}\sum^{N-1}_{k=0}\sum_{x\in B_{k}}\frac{P^{\omega}_{0}(D_{N}=k, X_{H_{k}}=x)}{P^{\omega}_{0}(D_{N}\leq N-1)}P^{\omega}_{y}(H_{N}\geq n) \\
& \qquad \geq
\frac{1}{4dn}\sum^{N-1}_{k=0}\sum_{x\in B_{k}}\frac{P^{\omega}_{0}(D_{N}=k, X_{H_{k}}=x)}{P^{\omega}_{0}(D_{N}\leq
N-1)}P^{\omega}_{y}(H_{N}\geq n).
\end{split}
\end{equation} 

If the trap $\mathfrak{P}_{N}$ retains enough the random walk~$X$, we will have $ H_{N}\geq n$, when it starts at $y$ (always the same $y=x+\epsilon(x)\hat{e}_{i_{0}}$ of the collection $\CC(x)$). Let
\begin{displaymath}
E_N:=\bigcup^{n-1}_{j=0}\left\{X_{j}\, \text{\textit{steps outside of the trap}} \,\mathfrak{P}_{N}\right\}
\end{displaymath}
and we say ``\textit{$X_{j}$ steps outside of the trap $\mathfrak{P}_{N}$ }", when $X_{j+1}$ is on a site of the border of the trap $\mathfrak{P}_{N}$,  i.e. $X_{j+1}=y\pm\hat{e}_{i}$, $\forall i\neq i_0$, or $X_{j+1}=x$ (resp. $X_{j+1}=z\pm\hat{e}_{i}$, $\forall i\neq i_0$, or $X_{j+1}=z+\epsilon (z)\hat{e}_{i_0}$) if $X_{j}=y$ (resp. if $X_j=z$). 

The complement of $E_N$ is in fact the event that $X$ does not leave the trap during its first $n$ jumps, i.e. $X$ jumps $n$ times, starting at $y$, in turn on $z$ and $y$, which, according to the configuration of the trap, costs for each jump a probability greater than 
$$
\frac{\xi}{\xi+(2d-1)N^{-\alpha}}.
$$  
Then, we have by the Markov property
$$
P^{\omega}_{y}(H_{N}\geq n)\geq P^{\omega}_{y}(E^c_N)\geq \left(\frac{\xi}{\xi+(2d-1)N^{-\alpha}}\right)^n,
$$
and since by the choice of $N^{\alpha}\leq n<(N+1)^{\alpha}$ 
$$
\left(\frac{\xi}{\xi+(2d-1)N^{-\alpha}}\right)^n \xrightarrow[n \to +\infty]{}  e^{-(2d-1)/\xi},
$$
it follows for all~$N$ large enough that
\begin{equation}
P^{\omega}_{y}(H_{N}\geq n)\geq\frac{ e^{-(2d-1)/\xi}}{2}.
\end{equation}
So, putting this in \eqref{if}, we obtain
\begin{eqnarray*}
P^{\omega}_{0}(X_{n}\in B_{N}|D_{N}\leq N-1) 
&\geq &
\frac{e^{-(2d-1)/\xi}}{8dn}\sum^{N-1}_{k=0}\sum_{x\in B_{N-1}}\frac{P^{\omega}_{0}(D_{N}=k, X_{H_{k}}=x)}{P^{\omega}_{0}(D_{N}\leq
N-1)}
\\
&\geq&
\frac{e^{-(2d-1)/\xi}}{8d n}.
\end{eqnarray*}
Now, according to \eqref{cantelli}, we have $P^{\omega}_{0}(D_{N}\leq N-1)\geq \ffrac{1}{2}$. Then we deduce
$$
P^{\omega}_{0}(X_{n}\in B_{N})\geq P^{\omega}_{0}(X_{n}\in B_{N}|D_{N}\leq N-1)P^{\omega}_{0}(D_{N}\leq N-1)\geq \frac{e^{-(2d-1)/\xi}}{16d n}.
$$
A fortiori, we have
$$
P^{\omega}_{0}(X_{n}\in B_{n^{1/\alpha}})\geq P^{\omega}_{0}(X_{n}\in B_N)\geq \frac{e^{-(2d-1)/\xi}}{16 d n}.
$$
Thus, for all $N\geq N_{0}$, by replacing the last inequality in \eqref{minun2}, we obtain
$$
P^{2n}_{\omega}(0,0)\geq \frac{\pi(0)\left(e^{-(2d-1)/\xi}/16d\right)^{2}7^{-d}}{n^{2+\delta(\gamma)}}.
$$
where $\delta(\gamma):=d(4d-2)\gamma/(1-\epsilon)$. When we let $\epsilon\longrightarrow 0$, we get \eqref{esup}.
\end{proofsect}

\section{\textbf{Standard heat-kernel decay}}
\label{shd}
We give here the proof of Theorem~\ref{thm}.

Let us first give some definitions and fix some notations besides those seen before.

Consider a Markov chain on a countable state-space~$V$ with transition probability denoted by $\cmss P(x,y)$ and invariant measure denoted by~$\pi$. Define~$\cmss Q(x,y)=\pi(x)\cmss P(x,y)$ and for each~$S_1,S_2\subset V$, let
\begin{equation}
\label{QSS}
\cmss Q(S_1,S_2)=\sum_{x\in S_1}\sum_{y\in S_2}\cmss Q(x,y).
\end{equation}
For each~$S\subset V$ with~$\pi(S)\in(0,\infty)$ we define
\begin{equation}
\label{PhiS}
\Phi_S=\frac{\cmss Q(S,S^c)}{\pi(S)}
\end{equation}
and use it to define the isoperimetric profile
\begin{equation}
\label{Phi-inf}
\Phi(r)=\inf\bigl\{\Phi_S\colon \pi(S)\le r\bigr\}.
\end{equation}
(Here~$\pi(S)$ is the measure of~$S$.)
It is easy to check that we may restrict the infimum to sets~$S$ that are connected in the graph structure induced on~$V$ by $\cmss P$.

To prove Theorem \ref{thm}, we combine basically two facts. On the one hand, we use Theorem~2 of Morris and Peres~\cite{MP} that we summarize here~: Suppose that~$\cmss P(x,x)\ge\sigma$ for some~$\sigma\in(0,1/2]$ and all~$x\in V$. Let~$\epsilon>0$ and~$x,y\in V$. Then
\begin{equation}
\label{MP-bound}
\cmss P^n(x,y)\le\epsilon\pi(y)
\end{equation}
for all~$n$ such that
\begin{equation}
\label{LK-bound}
n\ge 1+\frac{(1-\sigma)^2}{\sigma^2}\int_{4[\pi(x)\wedge\pi(y)]}^{4/\epsilon}\frac4{u\Phi(u)^2}\,\text d u.
\end{equation}
 Let $B_{N+1}=[-(N+1),N+1]^d$ and $\BB_{N+1}$ denote the set of nearest-neighbor bonds of $B_{N+1}$, i.e., $\BB_{N+1}=\{b=(x,y): x,y\in B_{N+1}, x\sim y\}$. Call $\Z^d_e$ the set of even points of $\Z^d$, i.e., the points $x:=(x_1,\ldots,x_d)$ such that $\vert\sum^{d}_{i=1}x_i\vert=2k$, with $k\in\N$  ($0\in \N$), and equip it with the graph structure
defined by~: two points $x,y\in \Z^d_e\subset\Z^d$ are neighbors when they are separated in $\Z^d$ by two steps, i.e. 
$$
\sum^{d}_{i=1}\vert x_i-y_i\vert=2.
$$
We operate the following modification on the environment~$\omega$ by defining $\tilde{\omega}_b=1$ on every bond $b\notin\BB_{N+1}$ and $\tilde{\omega}_b=\omega_b$ otherwise. Then, we will adapt the machinery above to the following setting
\begin{equation}
V=\Z^d_e,\quad\cmss P= P^2_{\tilde{\omega}}\quad\text{and}\quad\pi=\pi_{\tilde{\omega}},
\end{equation}
with the objects in \twoeqref{QSS}{Phi-inf} denoted by~$\cmss Q_{\tilde{\omega}}$, $\Phi_S^{({\tilde{\omega}})}$ and~$\Phi_{\tilde{\omega}}(r)$. So, the random walk associated with $P^2_{\tilde{\omega}}$ moves on the even points. 

\smallskip
On the other hand, we need to know the following standard fact that gives a lower bound of the conductances of the box $B_{N}$. For a proof, see \cite{Fontes-Mathieu}, Lemma~3.6.
\begin{lemma}
\label{L}
Under assumption~\eqref{1},
\begin{equation}
\label{LL}
\lim_{N\rightarrow+\infty}\frac{\log\inf_{b\in\BB_{N}}\omega_b}{\log N}=-\frac{d}{\gamma},\qquad \Q-a.s.
\end{equation}
\end{lemma}
Thus, for arbitrary $\mu>0$, we can write $\Q-$a.s.,  for all $N$ large enough
\begin{equation}
\label{mu}
\inf_{b\in\BB_{N+1}}\omega_b\geq N^{-(\frac{d}{\gamma}+\mu)}.
\end{equation}

%
%

Our next step involves extraction of appropriate bounds on surface and volume terms. 
\begin{lemma}
\label{lemma-adapt}
Let~$d\ge2$ and set $\alpha(N):=N^{-(\frac{d}{\gamma}+\mu)}$, for arbitrary $\mu>0$. Then,  for a.e. $\omega$, there exists a constant~$c>0$  such that  the following holds: For $N$ large enough and any finite connected~$\Lambda\subset \Z^d_e$,
we have
\begin{equation}
\label{Q-actual}
\cmss Q_{\tilde{\omega}}(\Lambda,\Z^d_e\setminus\Lambda)\ge
c \alpha(N)^2\pi_{\tilde{\omega}}(\Lambda)^{\frac{d-1}d}.
\end{equation}
\end{lemma}

The proof of lemma \ref{lemma-adapt} will be a consequence of the following well-known fact of isoperimetric inequalities on $\Z^d$ (see \cite{Woess}, Chapter I, \S~4). For any connected~$\Lambda\subset\Z^d$, let~$\partial\Lambda$ denote the set of edges between~$\Lambda$ and~$\Z^d\setminus\Lambda$. Then, there exists a constant $\kappa$ such that
\begin{equation}
\label{ii}
|\partial\Lambda|\ge \kappa|\Lambda|^{\frac{d-1}{d}}
\end{equation}
for every finite connected $\Lambda\subset\Z^d$. This remains true for $\Z^d_e$.  

\begin{proofsect}{Proof of lemma~\ref{lemma-adapt}}
For some arbitrary $\mu>0$, set $\alpha:=\alpha(N)=N^{-(\frac{d}{\gamma}+\mu)}$ and let~$N\gg1$. For any finite connected~$\Lambda\subset \Z^d_e$, we claim that
\begin{equation}
\label{Q-bd}
{\cmss Q}_{\tilde{\omega}}(\Lambda,\Z^d_e\setminus\Lambda)\ge \frac{\alpha^2}{2d}\,| \partial\Lambda|
\end{equation}
and
\begin{equation}
\label{vol-bd}
\pi_{\tilde{\omega}}(\Lambda)\le 2d|\Lambda|.
\end{equation}
Then, Lemma~\ref{L} gives a.s. $\inf_{b\in \BB_N}\omega(b)>\alpha$ and  by virtue of \eqref{ii}, we have $|\partial\Lambda|\ge \kappa|\Lambda|^{\frac{d-1}{d}}$, then~\eqref{Q-actual} will follow from \twoeqref{Q-bd}{vol-bd}.

It remains to prove \twoeqref{Q-bd}{vol-bd}. The bound \eqref{vol-bd} is implied by~$\pi_{\tilde{\omega}}(x)\le2d$. For \eqref{Q-bd}, since~$P^2_\omega$ represents two steps of a random walk, we get a lower bound on~$\cmss Q_\omega(\Lambda,\Z^d_e\setminus\Lambda)$ by picking a site~$x\in\Lambda$ which has a neighbor~$y\in\Z^d$ that has a neighbor~$z\in\Z^d_e$ on the outer boundary of~$\Lambda$. By Lemma~\ref{L}, if $x$ or $z\in B_{N+1}$,    
the relevant contribution is bounded by
\begin{equation}
\label{aa}
\pi_{\tilde{\omega}}(x) P^2_{\tilde{\omega}}(x,z)\ge\pi_{\tilde{\omega}}(x)\frac{\tilde{\omega}_{xy}}{\pi_{\tilde{\omega}}(x)}\frac{\tilde{\omega}_{yz}}{\pi_{\tilde{\omega}}(y)}\ge\frac{\alpha^2}{2d}.
\end{equation}
For the case where $x,z\notin\Z^d_e\cap B_{N+1}$, clearly the left-hand side of \eqref{aa} is bounded by $1/(2d)>\alpha^{2}/(2d)$.
Once~$\Lambda$ has at least two elements, we can do this for~$(y,z)$ ranging over all bonds in~$\partial\Lambda$, so summing over $(y,z)$ we get~\eqref{Q-bd}.
\end{proofsect}

Now we get what we need to estimate the decay of $P^{2n}_\omega(0,0)$.

\begin{proofsect}{Proof of Theorem~\ref{thm}}
Let $d\geq5$,  $\gamma>8d$ and choose $\mu>0$ such that
$$
\mu<\frac{1}{8}-\frac{d}{\gamma}.
$$
Let $n=\lfloor N/2\rfloor$, $N\gg1$, and consider the random walk on $\tilde{\omega}$.

We will derive a bound on~$\Phi_\Lambda^{({\tilde{\omega}})}$ for connected~$\Lambda\subset \Z^d_e$. Henceforth~$c$ denotes a generic constant.
Observe that \eqref{Q-actual} implies
\begin{equation}
\Phi_\Lambda^{({\tilde{\omega}})}\ge c\alpha^2\pi_{\tilde{\omega}}(\Lambda)^{-1/d}.
\end{equation}
Then, we conclude that
\begin{equation}
\Phi_{\tilde{\omega}}(r)\ge  c \alpha^2r^{-1/d}
\end{equation} 

The relevant integral is thus bounded by
\begin{eqnarray}
\frac{(1-\sigma)^2}{\sigma^2}\int_{4[\pi(0)\wedge \pi(x)]}^{4/\epsilon}\frac{4}{u\Phi_{\tilde{\omega}}(u)^2}\,\text d u
&\le&
 c\alpha^{-4}\sigma^{-2}\epsilon^{-2/d}
\end{eqnarray}
for some constant~$c>0$. Setting~$\epsilon$ proportional to
$n^{\frac{4d^2}{\gamma}+4\mu d-\frac{d}{2}}$, and noting \mbox{$\sigma\ge\alpha^2/(2d)$}, the right-hand side
is less than~$n$ and by setting $\delta(\gamma)=4d^2/\gamma$, we will get 
\begin{equation}
\label{bb}
P^{2n}_{\tilde{\omega}}(0,x)\leq \frac{c}{n^{\frac{d}{2}-\delta(\gamma)-4\mu d}},\qquad  \forall x\in \Z^d_e.
\end{equation}
As the random walk will not leave the box $B_N$ by time $2n$, we can replace ${\tilde{\omega}}$ by $\omega$ in \eqref{bb}, and since $P^{2n}_\omega(0,x)=0$ for each $x\notin B_N$, then after letting $\mu\rightarrow0$, we get
$$
\limsup_{n\rightarrow+\infty}\sup_{x\in\Z^d}\frac{\log P^{2n}_{\omega}(0,x)}{\log n}\leq -\frac{d}{2}+\delta(\gamma).
$$
This proves the claim for even~$n$;
for odd~$n$ we just concatenate this with a single step of the
random walk. 
\end{proofsect}

\section*{Acknowledgments}
\noindent 
I express my gratitude to my father Youcef Bey. I wish to thank my Ph.D. advisor, Pierre Mathieu for suggesting and discussions on this problem, and Abdelatif Bencherif-Madani for his support. I also would like to thank the referees for their careful reading and comments that led to an improvement of the paper.

\bigskip


\end{document}